\documentclass[reqno,12pt,letterpaper]{amsart}
\usepackage{amsmath,amssymb,amsthm,graphicx,mathrsfs,url,mathtools}
\usepackage[usenames,dvipsnames]{xcolor}
\usepackage{hyperref}
\hypersetup{colorlinks=true,linkcolor=Red,citecolor=Green}
\usepackage{amsxtra}

\def\?[#1]{\textbf{[#1]}\marginpar{\Large{\textbf{??}}}}

\let\epsilon=\varepsilon % sorry Knuth
\let\phi=\varphi

\newcommand{\D}{{\mathcal D}}

\newcommand{\RR}{{\mathbb R}}

\newcommand{\UU}{{\mathcal U}}

\newcommand{\CC}{{\mathbb C}}
\newcommand{\CI}{{{\mathcal C}^\infty}}

\newcommand{\wbar}{\overline{w}}
\newcommand{\xbar}{\overline{x}}
\newcommand{\ybar}{\overline{y}}

\newtheorem{thm}{Theorem}
\newtheorem{prop}{Proposition}%[section]

\numberwithin{equation}{section}
\numberwithin{prop}{section}
\numberwithin{thm}{section}

\DeclareMathOperator{\Hol}{Hol}

\let\Im=\Imag

\let\Re=\Real

\DeclareMathOperator{\Top}{Top}

\DeclareMathOperator{\vc}{vc}

\title{Boundedness of metaplectic Toeplitz operators and Weyl symbols}
\author{Haoren Xiong}
\email{haorenxiong@math.ucla.edu}
\address{Department of Mathematics, University of California,
Los Angeles, CA 90095, USA}
\begin{document}

\begin{abstract}
We study Toeplitz operators on the Bargmann space, whose Toeplitz symbols are exponentials of complex inhomogeneous quadratic polynomials. Extending a result by Coburn--Hitrik--Sj\"{o}strand \cite{coburn2023characterizing}, we show that the boundedness of such Toeplitz operators implies the boundedness of the corresponding Weyl symbols, thus completing the proof of the Berger--Coburn conjecture in this case. We also show that a Toeplitz operator is compact precisely when its Weyl symbol vanishes at infinity in this case.
\end{abstract}

\maketitle

\section{Introduction and statement of results}
\label{sec:intro}

The Berger--Coburn conjecture \cite{berger1994heat}, \cite{Coburn2019fock}, has been a long standing open problem in the theory of Toeplitz operators, claiming the equivalence between the boundedness of a Toeplitz operator on the Bargmann space and the boundedness of its Weyl symbol. In \cite{coburn2019positivity}, Coburn--Hitrik--Sj\"{o}strand have built connections between Toeplitz operators and metaplectic Fourier integral operators (FIOs) in the complex domain, by considering positivity of canonical transformations. The point of view of complex FIOs was used in \cite{coburn2019positivity} to show that the boundedness of a special class of Toeplitz operators whose Toeplitz symbols are exponentials of complex quadratic forms, is implied by the boundedness of their Weyl symbols, thus proving the sufficiency part of the conjecture in this case. In a follow up paper \cite{coburn2021weyl}, the proof of the sufficiency part of the conjecture was extended to the case of Toeplitz symbols given by exponentials of inhomogeneous quadratic polynomials. In a more recent work \cite{coburn2023characterizing}, the necessity part of the Berger--Coburn conjecture for Toeplitz operators with symbols that are exponentials of complex quadratic forms has been established, showing that the boundedness of such Toeplitz operators implies the boundedness of their Weyl symbols. The purpose of this work is to extend the result of \cite{coburn2023characterizing} to the case where Toeplitz symbols are exponentials of inhomogeneous quadratic polynomials. In this section, we shall introduce the assumptions and state the main results of this work.

\noindent
Let $\Phi_0$ be a strictly plurisubharmonic quadratic form on $\CC^n$ and let us define
\begin{equation}
\label{eqn:LambdaPhi0}
    \Lambda_{\Phi_0} = \bigg\{ \bigg(x,\frac{2}{i}\frac{\partial\Phi_0}{\partial x}(x)\bigg);\ x\in\CC^n \bigg\}\subset\CC_x^n\times\CC_\xi^n =\CC^{2n}.
\end{equation}
The real linear subspace $\Lambda_{\Phi_0}$ is I-Lagrangian and R-symplectic, in the sense that the restriction of the complex symplectic form on $\CC^{2n}$ to $\Lambda_{\Phi_0}$ is real and non-degenerate.

\noindent
Associated to the quadratic form $\Phi_0$ is the Bargmann space
\begin{equation}
\label{eqn:Bargmann}
    H_{\Phi_0}(\CC^n) = L^2(\CC^n, e^{-2\Phi_0}L(dx))\cap\Hol(\CC^n),
\end{equation}
where $L(dx)$ is the Lebesgue measure on $\CC^n$. We denote the orthogonal projection by
\begin{equation}
\label{eqn:orthoProj}
    \Pi_{\Phi_0} : L^2(\CC^n, e^{-2\Phi_0}L(dx)) \to H_{\Phi_0}(\CC^n).
\end{equation}
Let $q$ be an inhomogeneous quadratic polynomial on $\CC^n$ with complex coefficients. We shall assume that
\[
\Re q_2(x) < \Phi_0^{\text{herm}}(x): = (\Phi_0(x)+\Phi_0(ix))/2,\quad x\neq 0,
\]
where $q_2$ is the principal part of $q$, i.e. a homogeneous polynomial of degree $2$ in $x$, $\xbar$.

\noindent
In this note, following \cite{coburn2023characterizing} and \cite{coburn2019positivity}, we shall be concerned with Toeplitz operators:
\[
    \Top(e^q) = \Pi_{\Phi_0}\circ e^q \circ \Pi_{\Phi_0} : H_{\Phi_0}(\CC^n) \to H_{\Phi_0}(\CC^n).
\]
Such operators can also be represented using the Weyl quatization
\begin{equation}
\label{eqn:Top=Weyl}
    \Top(e^q) = a^\text{w} (x,D_x),
\end{equation}
see \cite[Chapter 13]{zworski2012semiclassical}, where the Weyl symbol $a\in\CI(\Lambda_{\Phi_0})$ is given by
\begin{equation}
\label{eqn:Weyl symbol}
    a\bigg(x,\frac{2}{i}\frac{\partial\Phi_0}{\partial x}(x)\bigg) = \bigg( \exp\bigg( \frac{1}{4}(\Phi_{0,x\xbar}'')^{-1} \partial_x\cdot\partial_{\xbar} \bigg) e^q\bigg)(x),\quad x\in\CC^n.
\end{equation}
Now we state the first main result of this work.
\begin{thm}
\label{thm:1}
Let $\Phi_0$ be a strictly plurisubharmonic quadratic form on $\CC^n$ and let $q$ be an inhomogeneous complex valued quadratic polynomial on $\CC^n$. Assume that the principal part $q_2$ of $q$ satisfies
\begin{equation}
\label{assumption 1}
\Re q_2(x) < \Phi_{\emph{herm}}(x): = (\Phi_0(x)+\Phi_0(ix))/2,\quad x\neq 0,
\end{equation}
and that
\begin{equation}
\label{assumption 2}
    \det\partial_x\partial_{\xbar}(2\Phi_0 - q_2) \neq 0.
\end{equation}
Then the Toeplitz operator $\Top(e^q):H_{\Phi_0}(\CC^n)\to H_{\Phi_0}(\CC^n)$ is bounded if and only if the corresponding Weyl symbol $a\in\CI(\Lambda_{\Phi_0})$ given by \eqref{eqn:Weyl symbol} satisfies $a\in L^\infty (\Lambda_{\Phi_0})$.
\end{thm}

\noindent
\emph{Remark.} The sufficiency of the boundedness of the Weyl symbol of $\Top(e^q)$ for the boundedness of the Toeplitz operator has been established in \cite{coburn2021weyl}, and here we only need to check the necessity.

\noindent
The compactness of Toeplitz operators $\Top(e^q)$ where $q$ is an inhomogeneous quadratic polynomial can also be characterized by the vanishing of their Weyl symbols at infinity, which agrees with a general conjecture stated in \cite{bauer2010heatflow}.

\begin{thm}
\label{thm:2}
Let $\Phi_0$ be a strictly plurisubharmonic quadratic form on $\CC^n$ and let $q$ be an inhomogeneous complex valued quadratic polynomial on $\CC^n$ whose principal part $q_2$ satisfies \eqref{assumption 1}, \eqref{assumption 2}. Then the Toeplitz operator $\Top(e^q):H_{\Phi_0}(\CC^n)\to H_{\Phi_0}(\CC^n)$ is compact if and only if the corresponding Weyl symbol $a\in\CI(\Lambda_{\Phi_0})$ given by \eqref{eqn:Weyl symbol} vanishes at infinity.
\end{thm}

The paper is organized as follows. In Section \ref{sec:Toeplitz}, we follow \cite{coburn2023characterizing} and study the action of bounded Toeplitz operator on a family of normalized coherent states on the Bargmann space. In Section \ref{sec:canonical trans}, we explore properties of the complex affine canonical transformation associated to the Toeplitz operator, using the information we obtained in the previous section. In Section \ref{sec:Weyl symbols}, we prove Theorem \ref{thm:1} by considering the canonical transformation defined using the Weyl symbol of the Toeplitz operator which is identical to the canonical transformation studied in the previous section. Theorem \ref{thm:2} is established in Section \ref{sec:compactness}. In Section \ref{sec:an example} we explore an explicit family of metaplectic Toeplitz operators on a quadratic Bargmann space, illustrating our theorems.

\medskip

\noindent
{\sc Acknowledgments.} The author would like to thank Michael Hitrik for suggesting this project and for many helpful discussions. The author would also like to thank Lewis Coburn for suggesting adding an explicit example -- see Section \ref{sec:an example}.

\section{Bounded Toeplitz operators}
\label{sec:Toeplitz}

Let $\Phi_0$ be a strictly plurisubharmonic quadratic form on $\CC^n$ and let $q$ be an inhomogeneous complex valued quadratic polynomial on $\CC^n$ whose principal part $q_2$ satisfies \eqref{assumption 1}, \eqref{assumption 2}. Following an argument in \cite[Section 4]{coburn2019positivity}, we see that when equipped with the maximal domain
\begin{equation}
\label{eqn:Top domain}
    \D(\Top(e^q)) = \{ u\in H_{\Phi_0}(\CC^n) ; e^q u\in L^2(\CC^n, e^{-2\Phi_0}L(dx)) \},
\end{equation}
the Toeplitz operator
\begin{equation}
\label{eqn:Toeplitz}
    \Top(e^q) = \Pi_{\Phi_0}\circ e^q \circ \Pi_{\Phi_0} : H_{\Phi_0}(\CC^n) \to H_{\Phi_0}(\CC^n).
\end{equation}
is densely defined.

\noindent
Let $a\in\CI(\Lambda_{\Phi_0})$ be the Weyl symbol of $\Top(e^q)$, given in \eqref{eqn:Weyl symbol} and let us recall that the implication $a\in L^\infty (\Lambda_{\Phi_0}) \implies \Top(e^q) \in \mathcal{L}(H_{\Phi_0}(\CC^n),H_{\Phi_0}(\CC^n))$ has been proved in \cite{coburn2021weyl}. Therefore, we shall only be concerned with the implication: $\Top(e^q)\in\mathcal{L}(H_{\Phi_0}(\CC^n),H_{\Phi_0}(\CC^n))\implies a\in L^\infty (\Lambda_{\Phi_0})$. For that, following \cite{coburn2021weyl}, let us recall the integral representation for $\Pi_{\Phi_0}$ and write for $u\in\D(\Top(e^q))$,
\begin{equation}
\label{eqn:Top integral rep.}
    \Top(e^q)u(x) = C\iint_\Gamma e^{2(\Psi_0(x,\theta)-\Psi_0(y,\theta))+Q(y,\theta)} u(y) dy d\theta,\quad C\neq 0,
\end{equation}
where $\Gamma$ is the contour in $\CC^{2n}$, given by $\theta=\ybar$, and $\Psi_0$ and $Q$ are the polarizations of $\Phi_0$ and $q$ respectively, i.e. a holomorphic quadratic form and a holomorphic quadratic polynomial on $\CC_{y,\theta}^{2n}$ such that $\Psi_0|_{\Gamma} = \Phi_0$, $Q|_{\Gamma} = q$.

\noindent
We now assume that
\begin{equation}
\label{Toeplitz is bounded}
    \Top(e^q)\in \mathcal{L}(H_{\Phi_0}(\CC^n), H_{\Phi_0}(\CC^n)).
\end{equation}
Following \cite{coburn2023characterizing}, we shall check first that it suffices to show that $a\in L^\infty (\Lambda_{\Phi_0})$, when the pluriharmonic part of $\Phi_0$ vanishes. For that we decompose
\[
    \Phi_0 = \Phi_{\text{herm}} + \Re f,\quad f(x) = (\Phi_0)_{xx}'' x\cdot x,
\]
and introduce the unitary metaplectic Fourier integral operator
\[
    \UU: H_{\Phi_0}(\CC^n) \owns u \mapsto e^{-f} u \in H_{\Phi_{\text{herm}}}(\CC^n).
\]
Let $\Pi_{\Phi_{\text{herm}}}$ be the orthogonal projection defined analogously to \eqref{eqn:orthoProj}, then we have
\[
    \Pi_{\Phi_{\text{herm}}} \circ e^q \circ \Pi_{\Phi_{\text{herm}}} = \UU\circ \Top(e^q) \circ \UU^{-1},
\]
and thus we have the following equivalence:
\[
    \eqref{Toeplitz is bounded} \iff \Pi_{\Phi_{\text{herm}}} \circ e^q \circ \Pi_{\Phi_{\text{herm}}} \in \mathcal{L}(H_{\Phi_{\text{herm}}}(\CC^n), H_{\Phi_{\text{herm}}}(\CC^n)).
\]
From the point of view of Weyl quantization, a direct computation shows that the exact Egorov theorem holds,
\[
    \UU \circ a^\text{w} (x,D_x) \circ \UU^{-1} = ((\kappa_A^{-1})^* a)^{\text{w}} (x,D_x),
\]
where $\kappa_A: \CC^{2n} \owns (y,\eta) \mapsto (y,\eta - Ay)\in \CC^{2n}$, $A = \frac{2}{i}(\Phi_0)_{xx}''$, and we have $\kappa_A(\Lambda_{\Phi_0}) = \Lambda_{\Phi_{\text{herm}}}$. We conclude the equivalence regarding the Weyl symbols:
\[
    a \in L^\infty (\Lambda_{\Phi_0}) \iff (\kappa_A^{-1})^* a \in L^\infty (\Lambda_{\Phi_{\text{herm}}}).
\]
In what follows, we shall assume thereby that $\Phi_0$ and $\Psi_0$ are given by
\begin{equation}
\label{eqn:Phi0}
    \Phi_0(x) = (\Phi_0)_{\xbar x}'' x\cdot\xbar,\quad x\in\CC^n,
\end{equation}
and
\begin{equation}
\label{eqn:Psi0}
    \Psi_0(x,y) = (\Phi_0)_{\xbar x}'' x\cdot y,\quad x, y \in\CC^n.
\end{equation}
It follows that
\begin{equation}
\label{eqn:decay diagonal}
    2\Re \Psi_0(x,\ybar) - \Phi_0(x) - \Phi_0(y) = - (\Phi_0)_{\xbar x}''(x-y)\cdot (\overline{x-y}) \asymp -|x-y|^2.
\end{equation}

\noindent
Assuming \eqref{Toeplitz is bounded}, we shall follow \cite{coburn2023characterizing} to consider the action of $\Top(e^q)$ on the space of ``coherent states", that is, the normalized reproducing kernels for the Bargmann space $H_{\Phi_0}(\CC^n)$. Let us introduce,
\begin{equation}
\label{eqn:kw(x)}
    k_w(x) := C_{\Phi_0} e^{2\Psi_0(x,\wbar)-\Phi_0(w)},\quad w\in\CC^n,
\end{equation}
where the constant $C_{\Phi_0}>0$ is chosen such that $k_w$ is normalized in $H_{\Phi_0}(\CC^n)$, i.e.
\begin{equation}
\label{eqn:kw normalized}
    \|k_w(\cdot)\|_{H_{\Phi_0}(\CC^n)} = 1,\quad w\in\CC^n.
\end{equation}
The properties \eqref{eqn:decay diagonal} and \eqref{assumption 1} guarantee that
\begin{equation}
\label{eqn:kw in domain}
    k_w \in \D(\Top(e^q)),\quad w\in\CC^n.
\end{equation}
Using \eqref{eqn:Top integral rep.} we have
\begin{equation}
\label{eqn:Top(eQ)kw integral}
    (\Top(e^q)k_w)(x) = C C_{\Phi_0} e^{-\Phi_0(w)} \iint_\Gamma e^{Q(y,\theta)-2\Psi_0(y,\theta)+2\Psi_0(x,\theta)+2\Psi_0(y,\wbar)} dy d\theta.
\end{equation}
Let us denote by $Q_2$ the polarization of the homogeneous part $q_2$ of $q$. It follows from \eqref{assumption 1} and \cite[Proposition 2.1]{coburn2023characterizing} that the holomorphic quadratic form
\begin{equation}
\label{eqn:quadratic form Q-2Psi}
    \CC_{y,\theta}^{2n} \owns (y,\theta) \mapsto Q_2(y,\theta) - 2\Psi_0(y,\theta)\ \,\text{is non-degenerate.}
\end{equation}
Then the method of quadratic stationary phase \cite[Lemma 13.2]{zworski2012semiclassical} shows that, with a new constant $C\neq 0$,
\begin{equation}
\label{eqn:Top(eQ)kw exact}
    (\Top(e^q)k_w)(x) = C e^{2f(x,\wbar)-\Phi_0(w)},
\end{equation}
where $f(x,z)$ is a holomorphic quadratic polynomial on $\CC_{x,z}^{2n}$ given by
\begin{equation}
\label{eqn:f(x,z) by vc}
    2f(x,z) = \vc_{y,\theta} ( Q(y,\theta)-2\Psi_0(y,\theta) + 2\Psi_0(x,\theta) + 2\Psi_0(y,z) ).
\end{equation}
Here ``vc" stands for the critical value. We shall also write $f(x,z)$ explicitly
\begin{equation}
\label{eqn:f(x,z) explicit}
    f(x,z) = \frac{1}{2}f_{xx}'' x\cdot x + f_{xz}''z\cdot x + \frac{1}{2} f_{zz}'' z\cdot z + l_x \cdot x + l_z \cdot z + f_0,
\end{equation}
and denote its quadratic part and linear part respectively by
\begin{equation}
\label{eqn:f_q, f_l}
    f_q(x,z):= \frac{1}{2}f_{xx}'' x\cdot x + f_{xz}''z\cdot x + \frac{1}{2} f_{zz}'' z\cdot z,\quad f_l(x,z):= l_x \cdot x + l_z \cdot z.
\end{equation}

\begin{prop}
\label{prop:conditions on f(x,z)}
Assume that \eqref{Toeplitz is bounded} holds, and let $f(x,z)$ be defined by \eqref{eqn:f(x,z) by vc}. Then the quadratic part $f_q$ and the linear part $f_l$ of $f$ satisfy
\begin{equation}
\label{eqn:f_q inequality}
    2\Re f_q(x,\wbar)\leq \Phi_0(x) + \Phi_0(w),\quad\forall (x,w)\in\CC_x^n\times\CC_w^n,
\end{equation}
and
\begin{equation}
\label{eqn:f_l condition}
    2\Re f_q(x,\wbar) = \Phi_0(x) + \Phi_0(w) \implies \Re f_l(x,\wbar) = 0.
\end{equation}
\end{prop}

\begin{proof}
In view of \eqref{eqn:kw in domain} and \eqref{eqn:Top(eQ)kw exact} we have
\begin{equation}
\label{eqn:e^f in Bargmann space}
    e^{2f(\cdot,\wbar)} \in H_{\Phi_0}(\CC^n),\quad \forall w\in\CC^n.
\end{equation}
In particular, \eqref{eqn:e^f in Bargmann space} implies that
\begin{equation}
\label{eqn:2Ref(x,0)}
    2\Re f_q(x,0) - \Phi_0(x) < 0,\quad 0\neq x\in\CC^n.
\end{equation}
By \eqref{eqn:Top(eQ)kw exact} we can write
\[
    \|\Top(e^q)k_w\|_{H_{\Phi_0}(\CC^n)}^2 = C^2 e^{-2\Phi_0(w)} \int_{\CC^n} e^{4\Re f(x,\wbar)- 2\Phi_0(x)} L(dx),
\]
applying the method of quadratic stationary (real) phase in view of \eqref{eqn:2Ref(x,0)}, we conclude that, with a new constant $C\neq 0$,
\begin{equation}
\label{eqn:Top(eQ)kw norm}
     \|\Top(e^q)k_w\|_{H_{\Phi_0}(\CC^n)}^2 = C^2 e^{-2\Phi_0(w)} e^{\sup_x(4\Re f(x,\wbar)-2\Phi_0(x))}.
\end{equation}
We note that $\|\Top(e^q)k_w\|_{H_{\Phi_0}(\CC^n)}$ is bounded uniformly in $w\in\CC^n$ as a result of \eqref{Toeplitz is bounded} and \eqref{eqn:kw normalized}. Thus there exists $C>0$ such that for all $w\in\CC^n$,
\[
    \sup_x(4\Re f(x,\wbar)-2\Phi_0(x)) - 2\Phi_0(w)\leq C,
\]
in other words,
\[
2\Re f(x,\wbar)-\Phi_0(x)-\Phi_0(w) \leq C,\quad \forall (x,w)\in\CC_x^n\times\CC_w^n.
\]
This is equivalent to \eqref{eqn:f_q inequality} and \eqref{eqn:f_l condition}.
\end{proof}

\noindent
The significance of the holomorphic quadratic polynomial $f(x,z)$ is that it gives an alternative way to represent the Fourier integral operator $\Top(e^q)$ in \eqref{eqn:Top integral rep.}. To see that let us first recall from \cite{hitrik2018minicourse}, \cite{sjostrand1996functionspaces} the integral representation for the orthogonal projection $\Pi_{\Phi_0}$ in \eqref{eqn:orthoProj}:
\begin{equation}
\label{eqn:Proj integral repres}
    \Pi_{\Phi_0} u(x) = a_0\iint e^{2\Psi_0(x,\ybar)-2\Phi_0(y)} u(y) dy d\ybar,\quad a_0\neq 0,
\end{equation}
then recalling \eqref{eqn:kw(x)} we shall apply $\Top(e^q)$ to
\[
    \D(\Top(e^q))\owns u(x) = \Pi_{\Phi_0} u(x) = \widetilde{a_0} \iint k_y(x) u(y)e^{-\Phi_0(y)} dy d\ybar
\]
we obtain by \eqref{eqn:Top(eQ)kw exact}, with a new constant $C\neq 0$,
\begin{equation}
\label{eqn:Top alt form}
\begin{split}
    \Top(e^q)u(x) &= \widetilde{a_0} \iint (\Top(e^q)k_y)(x) u(y)e^{-\Phi_0(y)} dy d\ybar \\
    &= C\iint e^{2f(x,\ybar)-2\Phi_0(y)}u(y) dy d\ybar = C\iint_\Gamma e^{2( f(x,z) - \Psi_0(y,z) )} u(y) dy dz.
\end{split}
\end{equation}
Here $\Gamma = \{ (y,z)\in\CC^{2n}; z=\ybar \}$. Therefore, the integral representations \eqref{eqn:Top integral rep.} and \eqref{eqn:Top alt form} should generate the same canonical transformation $\kappa$, see also \cite[Section 2]{coburn2023characterizing} for a direct calculation. In the next section we shall introduce a canonical transformation $\kappa$ associated with the phase function
\begin{equation}
\label{eqn:phi(x,y,z)}
    \phi(x,y,z) := \frac{2}{i}(f(x,z)-\Psi_0(y,z))
\end{equation}
in the representation \eqref{eqn:Top alt form} of the Fourier integral operator $\Top(e^q)$, and explore properties of $\kappa$ following from the conditions  \eqref{eqn:f_q inequality} and \eqref{eqn:f_l condition} satisfied by $f(x,z)$.

\section{Canonical transformation of the Toeplitz operator}
\label{sec:canonical trans}

In order to define a canonical transformation associated to \eqref{eqn:Top alt form}, we shall first prove a result similar to \cite[Proposition 2.2]{coburn2023characterizing} in our setting (inhomogeneous $f(x,z)$).
\begin{prop}
\label{prop:f_xz'' invertible}
The quadratic polynomial $f(x,z)$ defined in \eqref{eqn:f(x,z) by vc} satisfies
\begin{equation}
\label{eqn:det f_xz'' not zero}
    \det f_{xz}'' \neq 0.
\end{equation}
\end{prop}

\begin{proof}
Using \eqref{eqn:Psi0} and \eqref{eqn:f(x,z) by vc} we have
\begin{equation}
\label{eqn:f(x,z) by vc new}
    2f(x,z) = \vc_{y,\theta}(Q(y,\theta) - 2(\Phi_0)_{\xbar x}''y\cdot\theta + 2(\Phi_0)_{\xbar x}''x\cdot\theta + 2(\Phi_0)_{\xbar x}'' y\cdot z ),
\end{equation}
and denote by $(y,\theta) = ((y(x,z),\theta(x,z))$ the unique (thanks to \eqref{eqn:quadratic form Q-2Psi}) critical point achieving the critical value \eqref{eqn:f(x,z) by vc new}. Then
\begin{equation}
\label{eqn:f_x' and f_xz''}
    f_x'(x,z) = (\Phi_0)_{x\xbar}'' \,\theta(x,z),\quad f_{xz}'' = (\Phi_0)_{x\xbar}'' \,\theta_z'(x,z).
\end{equation}
It remains to show that $\theta_z'(x,z)$ is invertible. We note that the critical point $(y,\theta) =(y(x,z),\theta(x,z))$ satisfies
\begin{equation}
\label{eqn:critical point f(x,z)}
2(\Phi_0)_{\xbar x}'' y - Q_\theta'(y,\theta) = 2 (\Phi_0)_{\xbar x}'' x,\quad 2(\Phi_0)_{x \xbar}'' \theta - Q_y'(y,\theta) = 2 (\Phi_0)_{x\xbar}'' z.
\end{equation}
Let us write
\begin{equation}
\label{eqn:Q explicit}
    Q(y,\theta) = \frac{1}{2}Q_{yy}'' y\cdot y + Q_{y\theta}'' \theta\cdot y + \frac{1}{2}Q_{\theta\theta}'' \theta\cdot\theta + a\cdot y + b\cdot\theta,\quad a, b \in \CC^n
\end{equation}
Then the equations \eqref{eqn:critical point f(x,z)} can be rewritten as follows:
\[
    \begin{pmatrix}
    A_{11} & A_{12} \\ A_{21} & A_{22}
    \end{pmatrix}
    \begin{pmatrix}
    y \\ \theta
    \end{pmatrix}
    =
    \begin{pmatrix}
    2(\Phi_0)_{\xbar x}''x + b \\ 2(\Phi_0)_{x\xbar}''z + a
    \end{pmatrix}
\]
where the $2n \times 2n$ matrix
\[
    \mathcal{A} = \begin{pmatrix}
    A_{11} & A_{12} \\ A_{21} & A_{22}
    \end{pmatrix} = \begin{pmatrix}
    2(\Phi_0)_{\xbar x}'' - Q_{\theta y}'' & -Q_{\theta\theta}'' \\ -Q_{yy}'' & 2(\Phi_0)_{x\xbar}'' - Q_{y\theta}''
    \end{pmatrix}
\]
is invertible by \eqref{eqn:quadratic form Q-2Psi}. Let us set
\[
    \mathcal{B} = \mathcal{A}^{-1} = \begin{pmatrix}
    B_{11} & B_{12} \\ B_{21} & B_{22}
    \end{pmatrix}.
\]
The $n\times n$ block $A_{11}$ is invertible thanks to \eqref{assumption 2}, thus $B_{22}$ is invertible by the Schur's complement formula, see for instance \cite[Lemma 3.1]{sjostrand2007elementary}. It then follows that $\theta_z'(x,z) = 2 B_{22} (\Phi_0)_{x\xbar}''$ is invertible, and we conclude using  \eqref{eqn:f_x' and f_xz''} that \eqref{eqn:det f_xz'' not zero} holds.
\end{proof}

\noindent
Let us now introduce a canonical relation associated to the phase function $\phi(x,y,z)$ in \eqref{eqn:phi(x,y,z)}, with $z\in\CC^n$ viewed as the fiber variables,
\begin{equation}
\label{eqn:kappa defn}
    \kappa : \CC^{2n} \owns (y,-\phi_y'(x,y,z)) \mapsto (x,\phi_x'(x,y,z)) \in \CC^{2n},\quad \phi_z'(x,y,z)=0.
\end{equation}
More precisely, $\kappa$ is given by
\begin{equation}
\label{eqn:kappa using f(x,z)}
    \kappa : \left(  y,\frac{2}{i}(\Phi_0)_{x\xbar}'' z \right) \mapsto \left( x,\frac{2}{i}f_x'(x,z) \right),\quad f_z'(x,z) = (\Phi_0)_{\xbar x}'' y.
\end{equation}
It is natural to introduce another canonical relation $\kappa_q$ which is associated to the homogeneous part of $\phi(x,y,z)$, i.e. $\phi_q (x,y,z) := \frac{2}{i}(f_q(x,z) - \Psi_0(y,z))$,
\begin{equation}
\label{eqn:kappa_q}
\begin{gathered}
    \kappa_q : (y,-(\phi_q)_y'(x,y,z)) \mapsto (x,(\phi_q)_x'(x,y,z)),\quad (\phi_q)_z'(x,y,z)=0, \\
    \text{or}\ \kappa_q : \left(  y,\frac{2}{i}(\Phi_0)_{x\xbar}'' z \right) \mapsto \left( x,\frac{2}{i}(f_q)_x'(x,z) \right),\quad (f_q)_z'(x,z) = (\Phi_0)_{\xbar x}'' y .
\end{gathered}
\end{equation}
Recalling \eqref{eqn:f(x,z) explicit}, \eqref{eqn:f_q, f_l} and noting \eqref{eqn:det f_xz'' not zero}, we have more explicitly,
\begin{equation}
\label{eqn:kappa explicit}
    \kappa\left(  y,\frac{2}{i}(\Phi_0)_{x\xbar}'' z \right) = \left( x,\frac{2}{i}(f_{xx}''x + f_{xz}''z + l_x) \right),\quad x=(f_{zx}'')^{-1}((\Phi_0)_{\xbar x}''y - f_{zz}''z - l_z)
\end{equation}
and
\begin{equation}
\label{eqn:kappa_q explicit}
    \kappa_q \left(  y,\frac{2}{i}(\Phi_0)_{x\xbar}'' z \right) = \left( x,\frac{2}{i}(f_{xx}''x + f_{xz}''z) \right),\quad x = (f_{zx}'')^{-1}((\Phi_0)_{\xbar x}''y - f_{zz}''z).
\end{equation}
We see that $\kappa_q$ is a complex linear canonical transformation and that $\kappa$ is a complex affine canonical transformation. Moreover, $\kappa$ admits the following factorization
\begin{equation}
\label{eqn:kappa factorization}
    \kappa = \kappa_l \circ \kappa_q,
\end{equation}
where $\kappa_l$ is a complex phase space translation
\begin{equation}
\label{eqn:kappa_l}
    \kappa_l:\CC^{2n}\owns\rho = (x,\xi) \mapsto \rho + H_{m_l}\in\CC^{2n},\quad H_{m_l} = ((m_l)_\xi',-(m_l)_x')
\end{equation}
with a linear function $m_l$ on $\CC^{2n}$ given by
\begin{equation}
\label{eqn:m_l}
    m_l(x,\xi) := \frac{2}{i}(f_{xx}'' (f_{zx}'')^{-1} l_z - l_x)\cdot x - (f_{zx}'')^{-1} l_z \cdot\xi,\quad (x,\xi)\in \CC^n\times\CC^n.
\end{equation}
We recall from \cite[Proposition 3.2]{coburn2019positivity} that \eqref{eqn:f_q inequality} implies that the canonical transformation $\kappa_q$ is positive relative to $\Lambda_{\Phi_0}$. Applying \cite[Theorem 1.1]{coburn2019positivity}, we obtain
\begin{equation}
\label{eqn:Lambda_Phi1}
    \kappa_q(\Lambda_{\Phi_0}) = \Lambda_{\Phi_1},
\end{equation}
where $\Phi_1$ is a strictly plurisubharmonic quadratic form satisfying
\begin{equation}
\label{eqn:Phi1 less than Phi0}
    \Phi_1 \leq \Phi_0.
\end{equation}
We make the following observation on the intersection $\Lambda_{\Phi_1}\cap \Lambda_{\Phi_0}$:
\begin{prop}
\label{prop:m_l vanishes on intersection}
Let $\Lambda_{\Phi_1} = \kappa_q(\Lambda_{\Phi_0})$, with $\kappa_q$ given by \eqref{eqn:kappa_q} or \eqref{eqn:kappa_q explicit}. Then the complex linear function $m_l(x,\xi)$ defined in \eqref{eqn:m_l} satisfies
\begin{equation}
\label{eqn:m_l=0 on intersection}
    \Im m_l\big\lvert_{\Lambda_{\Phi_1}\cap \Lambda_{\Phi_0}} = 0 .
\end{equation}
\end{prop}

\begin{proof}
Recalling \eqref{eqn:LambdaPhi0} and \eqref{eqn:Phi0} we get
\begin{equation}
\label{eqn:LambdaPhi0 explicit}
    \Lambda_{\Phi_0} = \left\{ \left( x , \frac{2}{i}(\Phi_0)_{x\xbar}'' \,\xbar \right) ; x\in\CC^n \right\}.
\end{equation}
In view of \eqref{eqn:Lambda_Phi1} and \eqref{eqn:LambdaPhi0 explicit}, the points on the intersection $\Lambda_{\Phi_1}\cap \Lambda_{\Phi_0}$ are of the form:
\begin{equation}
\label{points on intersection}
    \Lambda_{\Phi_1}\cap \Lambda_{\Phi_0} \owns (x,\xi) = \left( x , \frac{2}{i}(\Phi_0)_{x\xbar}'' \,\xbar \right) = \kappa_q \left( w , \frac{2}{i}(\Phi_0)_{x\xbar}'' \,\wbar \right),\quad x,   w\in\CC^n.
\end{equation}
Using \eqref{eqn:kappa_q explicit}, we see that \eqref{points on intersection} $\implies (\Phi_0)_{x\xbar}'' \,\xbar = f_{xx}''x + f_{xz}''\wbar$. It then follows that for $(x,\xi)\in \Lambda_{\Phi_1}\cap \Lambda_{\Phi_0}$, by \eqref{eqn:m_l} we have
\[
\begin{split}
    m_l(x,\xi) &= \frac{2}{i}(f_{xx}'' (f_{zx}'')^{-1} l_z - l_x)\cdot x - (f_{zx}'')^{-1} l_z \cdot \frac{2}{i}((\Phi_0)_{x\xbar}'' \,\xbar) \\
    &= 2i\big( l_x \cdot x + (f_{zx}'')^{-1} l_z\cdot( (\Phi_0)_{x\xbar}'' \,\xbar - f_{xx}'' x ) \big) \\
    &= 2i\big( l_x \cdot x + (f_{zx}'')^{-1} l_z\cdot( f_{xz}''\wbar ) \big) = 2i(l_x\cdot x + l_z \cdot \wbar).
\end{split}
\]
Recalling \eqref{eqn:f_q, f_l}, we obtain that for $(x,\xi)$ satisfying \eqref{points on intersection},
\begin{equation}
\label{eqn:Im m_l on intersection}
    \Im m_l(x,\xi) = 2\Re (l_x\cdot x + l_z \cdot \wbar) = 2\Re f_l(x,\wbar).
\end{equation}
Arguing as in the proof of \cite[Proposition 3.1]{coburn2019positivity}, we see that
\begin{equation}
\label{eqn:graph of kappa_q}
    \left( x , \frac{2}{i}\frac{\partial\Phi_1}{\partial x}(x) \right) = \kappa_q \left( w , \frac{2}{i}\frac{\partial\Phi_0}{\partial w}(w) \right) \iff 2 \Re f_q (x,\wbar) = \Phi_1(x) + \Phi_0(w).
\end{equation}
We note that for $(x,\xi)\in \Lambda_{\Phi_1}\cap\Lambda_{\Phi_0}$, $\frac{\partial\Phi_1}{\partial x}(x) = \frac{\partial\Phi_0}{\partial x}(x)$, thus by homogeneity we get
\begin{equation}
\label{eqn:intersection Phi1=Phi0}
     \Phi_1(x) = \Phi_0(x),\quad x\in\pi_x (\Lambda_{\Phi_1}\cap\Lambda_{\Phi_0}).
\end{equation}
Combining \eqref{eqn:graph of kappa_q} and \eqref{eqn:intersection Phi1=Phi0}, we have, for $x,w\in\CC^n$ satisfying \eqref{points on intersection},
\begin{equation}
\label{describe intersection}
    \Phi_0(x) + \Phi_0(w) - 2\Re f_q(x,\wbar) = 0.
\end{equation}
We then conclude \eqref{eqn:m_l=0 on intersection} from \eqref{eqn:Im m_l on intersection}, \eqref{describe intersection} and \eqref{eqn:f_l condition}.
\end{proof}

\section{Weyl symbols and proof of Theorem 1.1}
\label{sec:Weyl symbols}

Let us recall that our Toeplitz operators can be represented by Weyl quantizations
\begin{equation}
\label{eqn:Top=Weyl 2}
    \Top(e^q) = a^{\text{w}} (x,D_x),
\end{equation}
where $a\in\CI(\Lambda_{\Phi_0})$ is the Weyl symbol of the Toeplitz operator $\Top(e^q)$, given by \eqref{eqn:Weyl symbol}. We shall recall from \cite{coburn2019positivity} that
\begin{equation}
\label{eqn:a(x,xi) by integral}
    a(x,\xi) = C_{\Phi_0} \int_{\CC^n} e^{-4\Phi_0^{\text{herm}}(x-y) + q(y)} L(dy),\quad C_{\Phi_0} \neq 0,\quad (x,\xi)\in\Lambda_{\Phi_0},
\end{equation}
where the integral converges thanks to \eqref{assumption 1}. Recalling \eqref{eqn:Phi0} and \eqref{eqn:Psi0}, we obtain, with $\xi = \frac{2}{i}\frac{\partial\Phi_0}{\partial x} = \frac{2}{i} (\Phi_0)_{x\xbar}''\xbar$,
\begin{equation}
\label{eqn:a(x,xi) by integral 2}
\begin{split}
    a(x,\xi) &= C_{\Phi_0} e^{-2ix\cdot\xi} \int_{\CC^n} e^{-4\Phi_0(y) + q(y) + 4 \Psi_0(x,\ybar) + 4\Psi_0(y,\xbar)} L(dy) \\
    &= C_{\Phi_0} e^{-2ix\cdot\xi} \iint_{\Gamma} e^{Q(y,\theta)-4\Psi_0(y,\theta) + 4\Psi_0(x,\theta) + 2i y\cdot \xi} dy d\theta,
\end{split}
\end{equation}
where $\Gamma$ is the contour in $\CC^{2n}$, given by $\theta=\ybar$, and $\Psi_0$, $Q$ are the polarizations of $\Phi_0$, $q$ respectively. Applying \cite[Proposition 2.1]{coburn2023characterizing} together with \eqref{assumption 1}, we see that, with $Q_2$ denoting the quadratic part of $Q$,
\begin{equation}
\label{non-degeneracy Q-4Psi}
    \CC_{y,\theta}^{2n} \owns (y,\theta) \mapsto Q_2(y,\theta) - 4\Psi_0(y,\theta)\ \,\text{is non-degenerate.}
\end{equation}
The method of quadratic stationary phase \cite[Lemma 13.2]{zworski2012semiclassical} then shows that
\begin{equation}
\label{eqn:a(x,xi) by vc}
    a(x,\xi) = C e^{-2ix\cdot\xi + \vc_{y,\theta}( Q(y,\theta)-4\Psi_0(y,\theta) + 4\Psi_0(x,\theta) + 2i y\cdot \xi )}
\end{equation}
for some constant $C\neq 0$. We can thereby write
\begin{equation}
\label{eqn:a(x,xi) by F and alpha}
    a(x,\xi) = C e^{i(F(x,\xi)+\alpha(x,\xi))},\quad(x,\xi)\in\Lambda_{\Phi_0},
\end{equation}
where $F$ is a holomorphic quadratic form on $\CC^{2n}$ and $\alpha$ is a complex linear function on $\CC^{2n}$ such that
\begin{equation}
\label{eqn:F,alpha}
    F(x,\xi)+\alpha(x,\xi) = -2x\cdot\xi + \frac{1}{i} \vc_{y,\theta}( Q(y,\theta)-4\Psi_0(y,\theta)+ 4\Psi_0(x,\theta) + 2i y\cdot \xi ).
\end{equation}
We remark that \eqref{eqn:a(x,xi) by F and alpha} extends to $(x,\xi)\in\CC^{2n}$ since $\Lambda_{\Phi_0}$ is maximally totally real. We now write
\begin{equation}
\label{eqn: a Weyl quantization}
    a^\text{w} (x,D_x) u(x) = \frac{1}{(2\pi)^n} \iint e^{i( (x-y)\cdot\xi + F((x+y)/2,\xi) + \alpha((x+y)/2,\xi) )} u(y) dy d\xi,
\end{equation}
and following \cite{coburn2021weyl}, $a^\text{w} (x,D_x)$ can be viewed as a Fourier integral operator with a non-degenerate phase function in the sense of H\"{o}rmander, given by
\[
    \tilde{\phi}(x,y,\xi) = (x-y)\cdot\xi + F((x+y)/2,\xi) + \alpha((x+y)/2,\xi).
\]
Here $\tilde{\phi}$ defines a canonical relation, with $\xi\in\CC^n$ viewed as the fiber variables,
\begin{equation}
\label{eqn:tilde kappa defn}
    \Tilde{\kappa} : \CC^{2n} \owns (y,-\tilde{\phi}_y'(x,y,\xi)) \mapsto (x,\tilde{\phi}_x'(x,y,\xi)) \in \CC^{2n},\quad \tilde{\phi}_{\xi}'(x,y,\xi)=0.
\end{equation}
We recall from \cite[Section 2]{coburn2021weyl} that $\Tilde{\kappa}$ can be rewritten in the form
\begin{equation}
\label{eqn:tilde kappa in rho}
    \tilde\kappa : \rho + \frac{1}{2}H_{F+\alpha}(\rho) \mapsto \rho - \frac{1}{2} H_{F+\alpha}(\rho),\quad\rho = \left( \frac{x+y}{2} , \xi\right)\in\CC^{2n},
\end{equation}
where $H_{F+\alpha}(\rho) = (F_\xi'(\rho) + \alpha_\xi' , -F_x'(\rho) - \alpha_x')$ is the Hamilton vector field of $F+\alpha$. We make the following observation, suggested by \eqref{eqn:Top=Weyl 2}.

\begin{prop}
\label{prop:kappa = tilde kappa}
Let $\tilde\kappa$ be a canonical relation given by \eqref{eqn:tilde kappa in rho}, with $F$ and $\alpha$ satisfying \eqref{eqn:F,alpha}. Let $\kappa$ be the canonical transformation associated to $\Top(e^q)$, given by \eqref{eqn:kappa using f(x,z)}. Then
\begin{equation}
\label{eqn:tilde kappa = kappa}
    \tilde\kappa = \kappa.
\end{equation}
\end{prop}

\begin{proof}
Recalling the argument at the end of Section \ref{sec:Toeplitz}, we observe that $\kappa$ can be viewed as a canonical transformation generated by the phase function in \eqref{eqn:Top integral rep.}, that is
\[
\Phi(x,y,\theta) = \frac{2}{i} (\Psi_0(x,\theta)-\Psi_0(y,\theta)) + \frac{1}{i} Q(y,\theta),
\]
and $\kappa$ is then given by
\[
    \kappa : \CC^{2n}\owns (y,-\Phi_y'(x,y,\theta))\mapsto (x,\Phi_x'(x,y,\theta))\in\CC^{2n},\quad\Phi_\theta'(x,y,\theta) = 0.
\]
More explicitly, using \eqref{eqn:Psi0} we get
\begin{equation}
\label{eqn:kappa using Q}
    \kappa : \left( y,\frac{2}{i}(\Phi_0)_{x\xbar}''\theta - \frac{1}{i}Q_y'(y,\theta) \right) \mapsto \left( y - \frac{1}{2}((\Phi_0)_{\xbar x}'')^{-1}Q_\theta'(y,\theta),\frac{2}{i}(\Phi_0)_{x\xbar}''\theta \right).
\end{equation}
We now denote by $(y,\theta) = ((y(x,\xi),\theta(x,\xi))$ the unique (thanks to \eqref{non-degeneracy Q-4Psi}) critical point achieving the critical value in \eqref{eqn:a(x,xi) by vc}. Using \eqref{eqn:F,alpha} together with \eqref{eqn:Psi0} we then see that
\begin{equation}
\label{eqn:H_(F+alpha)}
    H_{F+\alpha} (x,\xi) = (2y - 2x, 4i(\Phi_0)_{x\xbar}''\theta + 2\xi).
\end{equation}
Letting $\rho = (x,\xi)\in\CC^{2n}$, then we have
\[
\rho+\frac{1}{2} H_{F+\alpha}(\rho) \mapsto \rho - \frac{1}{2} H_{F+\alpha}(\rho) \iff (y,2i(\Phi_0)_{x\xbar}''\theta + 2\xi)\mapsto (2x - y, -2i(\Phi_0)_{x\xbar}''\theta).
\]
This means that, in view of \eqref{eqn:tilde kappa in rho},
\begin{equation}
\label{eqn:tilde kappa y theta}
    \tilde\kappa : (y,2i(\Phi_0)_{x\xbar}''\theta + 2\xi)\mapsto (2x - y, -2i(\Phi_0)_{x\xbar}''\theta),\quad (y,\theta) = ((y(x,\xi),\theta(x,\xi)).
\end{equation}
We note that the critical point $(y,\theta) = (y(x,\xi),\theta(x,\xi))$ corresponding to the critical value in \eqref{eqn:F,alpha} satisfies
\begin{equation}
\label{eqn:critical point}
    Q_y'(y,\theta) - 4(\Phi_0)_{x\xbar}''\theta + 2i\xi = 0,\quad Q_\theta'(y,\theta) - 4(\Phi_0)_{\xbar x}'' y + 4(\Phi_0)_{\xbar x}'' x = 0.
\end{equation}
It follows that
\[
    2i(\Phi_0)_{x\xbar}''\theta + 2\xi = \frac{2}{i}(\Phi_0)_{x\xbar}''\theta - \frac{1}{i}Q_y'(y,\theta),\quad 2x - y = y - \frac{1}{2}((\Phi_0)_{\xbar x}'')^{-1}Q_\theta'(y,\theta).
\]
We can then rewrite \eqref{eqn:tilde kappa y theta}, with $(y,\theta) = (y(x,\xi),\theta(x,\xi))$, $(x,\xi)\in\CC^{2n}$,
\begin{equation}
\label{eqn:tilde kappa rewritten}
    \tilde\kappa : \left( y,\frac{2}{i}(\Phi_0)_{x\xbar}''\theta - \frac{1}{i}Q_y'(y,\theta) \right) \mapsto \left( y - \frac{1}{2}((\Phi_0)_{\xbar x}'')^{-1}Q_\theta'(y,\theta),\frac{2}{i}(\Phi_0)_{x\xbar}''\theta \right).
\end{equation}
Recalling \eqref{eqn:Q explicit}, the equations \eqref{eqn:critical point} take the form
\begin{equation}
\label{eqn:critical point 2}
    \widetilde{\mathcal{A}} \begin{pmatrix}
        y \\ \theta
    \end{pmatrix} = \begin{pmatrix}
    4(\Phi_0)_{\xbar x}''x + b \\ 2i\xi + a
    \end{pmatrix},\quad \widetilde{\mathcal{A}} = \begin{pmatrix}
    4(\Phi_0)_{\xbar x}'' - Q_{\theta y}'' & -Q_{\theta\theta}'' \\ -Q_{yy}'' & 4(\Phi_0)_{x\xbar}'' - Q_{y\theta}''
    \end{pmatrix}.
\end{equation}
Here $\widetilde{\mathcal{A}}$ is invertible thanks to \eqref{non-degeneracy Q-4Psi}, we then conclude that
\[
    \CC^{2n}\owns (x,\xi) \mapsto (y,\theta) = (y(x,\xi),\theta(x,\xi))\in\CC^{2n}\,\text{ is a bijectiive map},
\]
thus \eqref{eqn:tilde kappa = kappa} follows by comparing \eqref{eqn:kappa using Q} and \eqref{eqn:tilde kappa rewritten}
\end{proof}

\noindent
Introducing the fundamental matrix of $F$ as in \cite[Appendix B]{coburn2019positivity}
\[
    \mathcal{F} = \begin{pmatrix}
        F_{\xi x}'' & F_{\xi\xi}'' \\ -F_{xx}'' & -F_{x\xi}''
    \end{pmatrix},
\]
we see that Hamilton vector field of $F$ is given by $H_F(\rho) = \mathcal{F}\rho$. Let us rewrite \eqref{eqn:tilde kappa in rho}:
\begin{equation}
\label{eqn:tilde kappa in rho rewritten}
    \tilde\kappa : \left( 1+ \frac{1}{2}\mathcal{F}\right)\rho + \frac{1}{2} H_\alpha \mapsto \left( 1 - \frac{1}{2}\mathcal{F}\right)\rho - \frac{1}{2} H_\alpha .
\end{equation}
In view of Proposition \ref{prop:f_xz'' invertible}, $\kappa$ is a complex affine canonical transformation. Therefore, Proposition \ref{prop:kappa = tilde kappa} shows that $\tilde\kappa$ in the form \eqref{eqn:tilde kappa in rho rewritten} is a canonical transformation, thus bijective. In other words,
\begin{equation}
\label{eqn:I + or - F}
    \det \left( 1 \pm \frac{1}{2}\mathcal{F}\right) \neq 0,
\end{equation}
we then see that
\begin{equation}
\label{eqn:kappa F}
    \kappa_F : \CC^{2n}\owns\left( 1+ \frac{1}{2}\mathcal{F}\right)\rho \mapsto \left( 1 - \frac{1}{2}\mathcal{F}\right)\rho \in\CC^{2n}
\end{equation}
is a complex linear canonical transformation. Recalling the computation from \cite[Section 2]{coburn2021weyl}, we conclude that $\tilde\kappa$ can be factored as follows
\begin{equation}
\label{eqn:tilde kappa factorization}
    \tilde\kappa = \kappa_\alpha \circ \kappa_F,
\end{equation}
where $\kappa_F$ is given in \eqref{eqn:kappa F} and $\kappa_\alpha$ is a complex phase space translation given by
\begin{equation}
\label{eqn:kappa_alpha}
    \kappa_\alpha : \rho\mapsto \rho - \frac{1}{2} H_{\alpha\circ\kappa_F^{-1}+\alpha}.
\end{equation}
We note that there is a unique way to factor an affine transformation into a composition of a linear transformation and a translation. Therefore, Proposition \ref{prop:kappa = tilde kappa} implies that factorizations \eqref{eqn:kappa factorization}
and \eqref{eqn:tilde kappa factorization} are equal, more explicitly, recalling \eqref{eqn:kappa_l} and \eqref{eqn:kappa_alpha}, we obtain
\begin{equation}
\label{eqn:same factorization}
    \kappa_q = \kappa_F,\quad m_l = -\frac{1}{2} (\alpha\circ \kappa_F^{-1} + \alpha).
\end{equation}
We have remarked in Section \ref{sec:canonical trans} that $\kappa_q$ is positive relative to $\Lambda_{\Phi_0}$ due to \eqref{eqn:f_q inequality}, see \cite[Proposition 3.2]{coburn2019positivity}. Then \eqref{eqn:same factorization} shows that, with $\Lambda_{\Phi_1}$ defined in \eqref{eqn:Lambda_Phi1},
\begin{equation}
\label{eqn:kappa_F positive}
    \kappa_F \text{ is positive relative to }\Lambda_{\Phi_0},\quad \kappa_F(\Lambda_{\Phi_0}) = \Lambda_{\Phi_1}.
\end{equation}
It follows from \eqref{eqn:kappa_F positive} and \cite[Proposition B.1]{coburn2019positivity} together with \eqref{eqn:I + or - F} that
\begin{equation}
\label{eqn:Im F nonnegative}
    \Im F\big\lvert_{\Lambda_{\Phi_0}} \geq 0.
\end{equation}
In view of \eqref{eqn:kappa_F positive} and \eqref{eqn:Im F nonnegative}, we then recall from \cite[Proposition 2.3]{coburn2021weyl} that
\begin{equation}
\label{describe intersection 2}
    \Lambda_{\Phi_1} \cap \Lambda_{\Phi_0} = \left\{ \left(1-\frac{1}{2}\mathcal{F}\right)\rho ; \,\rho\in\Lambda_{\Phi_0},\ \Im F(\rho) = 0 \right\} \subset \CC^{2n}.
\end{equation}
Using \eqref{eqn:kappa F} and \eqref{eqn:same factorization} we compute directly to obtain
\begin{equation}
\label{eqn:alpha on intersection}
     m_l \left( \left(1-\frac{1}{2}\mathcal{F}\right)\rho \right) = -\frac{1}{2}(\alpha\circ\kappa_F^{-1} + \alpha)\left( \left(1-\frac{1}{2}\mathcal{F}\right)\rho \right) = -\alpha(\rho).
\end{equation}
We conclude from \eqref{eqn:m_l=0 on intersection}, \eqref{describe intersection 2}, and \eqref{eqn:alpha on intersection} that
\begin{equation}
\label{eqn:Im alpha=0 when Im F=0}
    \rho\in\Lambda_{\Phi_0},\ \Im F(\rho) = 0 \implies \Im \alpha (\rho) = 0.
\end{equation}
The expression \eqref{eqn:a(x,xi) by F and alpha} for the Weyl symbol, together with \eqref{eqn:Im F nonnegative} and \eqref{eqn:Im alpha=0 when Im F=0}, allow us to conclude that
\[
    a = C e^{i(F+l)} \in L^\infty (\Lambda_{\Phi_0}),
\]
which completes the proof of Theorem \ref{thm:1}.

\section{Compact Toeplitz operators and proof of Theorem 1.2}
\label{sec:compactness}

In this section we shall prove Theorem \ref{thm:2}. Let $\Phi_0$ be a strictly plurisubharmonic quadratic form on $\CC^n$ and let $q$ be an inhomogeneous complex valued quadratic polynomial on $\CC^n$ whose principal part $q_2$ satisfies \eqref{assumption 1}, \eqref{assumption 2}. We first verify that the compactness of the Toeplitz operator $\Top(e^q) : H_{\Phi_0}(\CC^n) \to H_{\Phi_0}(\CC^n)$ is implied by the vanishing of the corresponding Weyl symbol given in \eqref{eqn:Weyl symbol} at infinity.

\noindent
We recall from \eqref{eqn:a(x,xi) by F and alpha} that the Weyl symbol $a$ of $\Top(e^q)$ has the form
\[
    a(x,\xi) = C e^{i(F(x,\xi)+\alpha(x,\xi))},\quad(x,\xi)\in\Lambda_{\Phi_0},
\]
where $F$ is a holomorphic quadratic form on $\CC^{2n}$ and $\alpha$ is a complex linear function on $\CC^{2n}$. The vanishing of $a$ at infinity is equivalent to the ellipticity of $\Im F|_{\Lambda_{\Phi_0}}$:
\begin{equation}
\label{eqn:Im F elliptic}
    \Im F\left( x,\frac{2}{i}\frac{\partial\Phi_0}{\partial x}(x) \right) \asymp |x|^2,\quad x\in\CC^n
\end{equation}
It follows from \cite[Proposition B.1]{coburn2019positivity} that \eqref{eqn:Im F elliptic} holds if and only if the canonical transformation $\kappa_F$ defined in \eqref{eqn:kappa F} is strictly positive relative to $\Lambda_{\Phi_0}$.

\noindent
We now assume that the Weyl symbol $a$ vanishes at infinity, thus $\kappa_F$ is strictly positive relative to $\Lambda_{\Phi_0}$. According to \cite[Theorem 1.1]{coburn2019positivity} we then get
\begin{equation}
\label{eqn:LambdaPhi1 in section 5}
    \kappa_F(\Lambda_{\Phi_0}) = \Lambda_{\Phi_1},
\end{equation}
where $\Phi_1$ is a strictly plurisubharmonic quadratic form on $\CC^n$ such that
\begin{equation}
\label{eqn:Phi_1 inequality}
    \Phi_0(x) - \Phi_1(x) \asymp |x|^2,\quad x\in\CC^n .
\end{equation}
We now consider the I-Lagrangian R-symplectic affine plane $\tilde\kappa(\Lambda_{\Phi_0})$ with $\tilde\kappa$ given by \eqref{eqn:tilde kappa in rho}. Applying \cite[Lemma 2.2]{coburn2021weyl} together with \eqref{eqn:tilde kappa factorization}, \eqref{eqn:kappa_alpha}, \eqref{eqn:same factorization}, and \eqref{eqn:LambdaPhi1 in section 5} we have
\begin{equation}
\label{eqn:Lambda_Phi_2}
    \tilde\kappa (\Lambda_{\Phi_0}) = \Lambda_{\Phi_2},
\end{equation}
where $\Phi_2$ is a strictly plurisubharmonic quadratic polynomial on $\CC^n$ given by
\begin{equation}
\label{eqn:Phi_2}
    \Phi_2(x) = \Phi_1(x) + \Im \left(m_l \left( x,\frac{2}{i}\frac{\partial\Phi_1}{\partial x}(x)\right) \right).
\end{equation}
According to the general theory, see \cite{sjostrand1982singular}, \cite{caliceti2012quadratic}, we can conclude from \eqref{eqn:Lambda_Phi_2} that
\begin{equation}
\label{eqn:operator between Bargmann}
    \Top(e^q) = a^\text{w} (x,D_x) \in\mathcal{L} ( H_{\Phi_0}(\CC^n), H_{\Phi_2}(\CC^n) ),
\end{equation}
where
\[
    H_{\Phi_2}(\CC^n) = L^2(\CC^n, e^{-2\Phi_2}L(dx))\cap\Hol(\CC^n).
\]
Combing \eqref{eqn:Phi_1 inequality} with \eqref{eqn:Phi_2}, recalling from \eqref{eqn:m_l} that $m_l$ is a linear function, we see that there exists a constant $C>0$ such that
\begin{equation}
\label{eqn:Phi_2 inequality}
   C^{-1} |x|^2 - C  \leq \Phi_0(x) - \Phi_2(x) .
\end{equation}
It follows that the embedding: $H_{\Phi_2}(\CC^n) \to H_{\Phi_0}(\CC^n)$ is compact, thus by \eqref{eqn:operator between Bargmann} the Toeplitz operator $\Top(e^q) : H_{\Phi_0}(\CC^n) \to  H_{\Phi_0}(\CC^n)$ is compact.

\noindent
It remains to prove that the vanishing of the Weyl symbol at infinity is a necessary condition for the compactness of the Toeplitz operator $\Top(e^q)$ on $H_{\Phi_0}(\CC^n)$. Following \cite[Section 3]{coburn2023characterizing} we shall consider the compact operator $\Top(e^q)$ acting on the coherent states $k_w$, $w\in\CC^n$, defined in \eqref{eqn:kw(x)}. Let us recall a result from \cite[Lemma 3.1]{coburn2023characterizing}:
\begin{prop}
\label{prop:kw weak convergence}
Let $k_w$, $w\in\CC^n$ be defined by \eqref{eqn:kw(x)}. Then $k_w \to 0$ weakly in $H_{\Phi_0}(\CC^n)$, as $|w|\to \infty$.
\end{prop}

\noindent
The compactness of $\Top(e^q)$ with Proposition \ref{prop:kw weak convergence} implies that
\begin{equation}
\label{eqn:Top kw to zero}
    \| \Top(e^q) k_w\|_{H_{\Phi_0}(\CC^n)} \to 0,\quad\text{as }|w|\to\infty.
\end{equation}
Recalling \eqref{eqn:Top(eQ)kw norm} together with \eqref{eqn:f(x,z) explicit}, \eqref{eqn:f_q, f_l}, we claim that
\begin{equation}
\label{eqn:f_q strict inequality}
    2\Re f_q (x,\wbar) < \Phi_0(x) + \Phi_0(w),\quad (0,0)\neq (x,w)\in \CC^{2n}.
\end{equation}
We note that $\Top(e^q)\in\mathcal{L}(H_{\Phi_0}(\CC^n),H_{\Phi_0}(\CC^n))$ due to the compactness, thus Proposition \ref{prop:conditions on f(x,z)} shows that \eqref{eqn:f_q inequality} and \eqref{eqn:f_l condition} hold. In order to justify \eqref{eqn:f_q strict inequality}, We shall argue by contradiction. Suppose that \eqref{eqn:f_q strict inequality} does not hold, then by \eqref{eqn:f_q inequality} we have
\begin{equation}
\label{condition:x0,w0}
    \exists\, (0,0)\neq(x_0, w_0)\in\CC^{2n}\ \text{such that }2\Re f_q (x_0,\overline{w_0}) = \Phi_0(x_0) + \Phi_0(w_0).
\end{equation}
In view of \eqref{eqn:2Ref(x,0)}, we see that $w_0\neq 0$. It follows from \eqref{eqn:f_l condition} that
\begin{equation}
\label{condition 2:x0,w0}
    \Re f_l(x_0,\overline{w_0}) = 0.
\end{equation}
Observing that \eqref{condition:x0,w0} and \eqref{condition 2:x0,w0} are homogeneous, we have for $\lambda>0$,
\[
    2\Re f_q (\lambda x_0,\overline{\lambda w_0}) = \Phi_0(\lambda x_0) + \Phi_0(\lambda w_0),\quad \Re f_l(\lambda x_0,\overline{\lambda w_0}) = 0.
\]
Therefore, recalling \eqref{eqn:f(x,z) explicit}, we get
\begin{equation}
\label{eqn:sup inequality}
    \Phi_0(\lambda w_0) + 2\Re f_0 = 2\Re f(\lambda x_0,\overline{\lambda w_0}) - \Phi_0(\lambda x_0) \leq \sup_x(2\Re f(x,\overline{\lambda w_0}) - 2\Phi_0(x))
\end{equation}
We conclude from \eqref{eqn:Top(eQ)kw norm} and \eqref{eqn:sup inequality} that for any $\lambda>0$,
\[
    \| \Top(e^q) k_{\lambda w_0} \|_{H_{\Phi_0}(\CC^n)} = C e^{-\Phi_0(\lambda w_0)} e^{\sup_x(2\Re f(x,\overline{\lambda w_0}) - 2\Phi_0(x))} \geq C e^{2\Re f_0} = \widetilde{C}>0.
\]
This contradicts \eqref{eqn:Top kw to zero} as $w_0\neq 0 \implies |\lambda w_0| \to \infty$ if we take $\lambda \to \infty$.

\noindent
A straightforward modification of \cite[Proposition 3.2]{coburn2019positivity} -- see \cite[Section 3]{coburn2023characterizing} for more details, shows that \eqref{eqn:f_q strict inequality} guarantees the strict positivity relative to $\Lambda_{\Phi_0}$ of the canonical transform $\kappa_q$ given in \eqref{eqn:kappa_q}. We then recall \eqref{eqn:same factorization} to conclude that $\kappa_F$ in \eqref{eqn:kappa F} is strictly positive relative to $\Lambda_{\Phi_0}$. Hence \eqref{eqn:Im F elliptic} follows, which implies the vanishing of the Weyl symbol $a\in\CI(\Lambda_{\Phi_0})$ at infinity. This completes the proof of Theorem \ref{thm:2}.

\section{An explicit example}
\label{sec:an example}

In this section we illustrate Theorem \ref{thm:1} and Theorem \ref{thm:2} by analyzing the boundedness and compactness of an explicit family of metaplectic Toeplitz operators on the Bargmann space $H_{\Phi_0}(\CC^n)$, for a model weight
\begin{equation}
\label{eqn:Phi0 example}
    \Phi_0(x) =\frac{|x|^2}{4},\quad x\in\CC^n,
\end{equation}
and its polarization
\begin{equation}
\label{eqn:Psi0 example}
    \Psi_0(x,y) = \frac{1}{4} x\cdot y,\quad x, y \in\CC^n.
\end{equation}
Inspired by \cite{coburn2021weyl} and \cite{coburn2023characterizing}, let us consider
\begin{equation}
\label{eqn:q example}
    q(x) = \lambda |x|^2 + A \xbar\cdot\xbar + \overline{c}\cdot x - d\cdot\xbar,\quad x\in\CC^n.
\end{equation}
where $\lambda\in\CC$, $c,d\in\CC^n$, and $A$ is a complex symmetric $n\times n$ matrix. We shall assume that $\Re\lambda + \|A\| < 1/4$, with $\|A\|$ denoting the operator norm:
\[
    \|A\| = \sup\{ |Ax| : x\in\CC^n,\ |x| = 1 \},\quad\text{where }|x| = (|x_1|^2 + \cdots + |x_n|^2)^{1/2}.
\]
It follows that conditions \eqref{assumption 1}, \eqref{assumption 2} are satisfied. Instead of computing the Weyl symbol of $\Top(e^q)$ using \eqref{eqn:a(x,xi) by integral}, it is more convenient to explore the complex affine canonical transformation $\kappa$ associated to $\Top(e^q)$. Let us recall \eqref{eqn:kappa using Q} to write
\begin{equation}
\label{eqn:kappa example pre}
    \kappa : \CC^{2n}\owns (y,\eta) \mapsto \left((1-2\lambda)y - \frac{8A(i\eta+\overline{c})}{1-2\lambda} + 2d , \frac{\eta - i\overline{c}}{1-2\lambda} \right) \in\CC^{2n}.
\end{equation}
Letting $\gamma = 1/(1-2\lambda)$, we can rewrite \eqref{eqn:kappa example pre} as
\begin{equation}
\label{eqn:kappa example}
    \kappa : \CC^{2n}\owns (y,\eta) \mapsto \left(\frac{y}{\gamma} - 8i\gamma A\eta - 8\gamma A\overline{c} + 2d , \gamma\eta - i\gamma\overline{c} \right)\in\CC^{2n}.
\end{equation}
We can factor $\kappa=\kappa_l\circ\kappa_q$ in the sense of \eqref{eqn:kappa factorization} with
\begin{equation}
\label{eqn:kappa_q example}
    \kappa_q : \CC^{2n}\owns (y,\eta) \mapsto \left(\frac{y}{\gamma} - 8i\gamma A\eta , \gamma\eta\right)\in\CC^{2n},
\end{equation}
\begin{equation}
\label{eqn:kappa_l,m_l example}
    \kappa_l: \CC^{2n}\owns\rho \mapsto \rho+ H_{m_l}\in\CC^{2n},\quad m_l(x,\xi) = i\gamma\overline{c}\cdot x + (2d-8\gamma A\overline{c})\cdot \xi.
\end{equation}
As we have seen in the proof of Theorem \ref{thm:1}, $\Top(e^q)$ is a bounded operator on $H_{\Phi_0}(\CC^n)$ if and only if
\begin{equation}
\label{eqn:characterize boundedness example}
    \text{$\kappa_q$ is positive relative to $\Lambda_{\Phi_0}$ and $\Im m_l|_{\kappa_q(\Lambda_{\Phi_0})\cap\Lambda_{\Phi_0}} = 0$.}
\end{equation}
Recalling the proof of \cite[Theorem 4.1]{coburn2023characterizing}, we see that
$\kappa_q$ in \eqref{eqn:kappa_q example} is positive relative to $\Lambda_{\Phi_0}$ if and only if
\begin{equation}
\label{eqn:A,lambda,condition 1}
    4|\gamma|^2 \|A\| \leq 1 - |\gamma|^2,\quad \gamma = \frac{1}{1-2\lambda}.
\end{equation}
We shall also study the intersection $\kappa_q(\Lambda_{\Phi_0}) \cap \Lambda_{\Phi_0}$. Recalling \eqref{eqn:LambdaPhi0} and \eqref{eqn:Phi0 example} we have
\begin{equation}
\label{eqn:LambdaPhi0 example}
    \Lambda_{\Phi_0} = \left\{ \left( y,\frac{\ybar}{2i} \right) ; y\in\CC^n \right\} \subset \CC_y^n\times \CC_\eta^n = \CC^{2n}.
\end{equation}
Combining \eqref{eqn:kappa_q example} and \eqref{eqn:LambdaPhi0 example} we see that
\begin{equation}
\label{eqn:intersection example}
    \left(y,\frac{\ybar}{2i} \right) \in \kappa_q(\Lambda_{\Phi_0}) \cap \Lambda_{\Phi_0} \iff \exists\, w\in\CC^n\text{ s.t. }\left(\frac{w}{\gamma} - 4\gamma A\wbar , \frac{\gamma\wbar}{2i}\right) = \left(y,\frac{\ybar}{2i} \right)
\end{equation}
Eliminating $w$, we can rewrite \eqref{eqn:intersection example} as follows
\begin{equation}
\label{eqn:intersection 2 example}
    \kappa_q(\Lambda_{\Phi_0}) \cap \Lambda_{\Phi_0} = \left\{ \left(y,\frac{\ybar}{2i} \right);\,(1-|\gamma|^2)y = 4|\gamma|^2 A\ybar,\ y\in\CC^n \right\}.
\end{equation}
It follows from \eqref{eqn:kappa_l,m_l example} and \eqref{eqn:intersection 2 example} that $\Im m_l|_{\kappa_q(\Lambda_{\Phi_0})\cap\Lambda_{\Phi_0}} = 0$ precisely when
\begin{equation}
\label{eqn:A,lambda,c,d condition 2}
    y\in\CC^n,\ (1-|\gamma|^2)y = 4|\gamma|^2 A\ybar \implies  \Re(\gamma\overline{c}\cdot y + (4\gamma A \overline{c}-d)\cdot\ybar) = 0.
\end{equation}
Therefore, we conclude that \eqref{eqn:characterize boundedness example} holds precisely when \eqref{eqn:A,lambda,condition 1} and \eqref{eqn:A,lambda,c,d condition 2} hold.

\noindent
Let us now recall the argument in Section \ref{sec:compactness} that $\Top(e^q)$ is compact on $H_{\Phi_0}(\CC^n)$ if and only if the linear canonical transformation $\kappa_q$ given in \eqref{eqn:kappa_q example} is strictly positive relative to $\Lambda_{\Phi_0}$. Arguing as in the proof of \cite[Theorem 4.1]{coburn2023characterizing}, we see that the strict positivity of $\kappa_q$ relative to $\Lambda_{\Phi_0}$ holds precisely when the inequality in \eqref{eqn:A,lambda,condition 1} is strict.

\noindent
We conclude this section with the following result.
\begin{thm}
\label{thm:3}
Let $\Phi_0 = |x|^2 /4$, $x\in\CC^n$. Let $\lambda\in\CC$, $c, d\in\CC^n$, and let $A$ be an complex symmetric $n\times n$ matrix such that $\Re\lambda + \|A\| < 1/4$. Let us set
\[
    q(x) = \lambda |x|^2 + A \xbar\cdot\xbar + \overline{c}\cdot x - d\cdot\xbar,\quad x\in\CC^n.
\]
The Toeplitz operator
\[
    \Top(e^q) = \Pi_{\Phi_0}\circ e^q \circ \Pi_{\Phi_0} : H_{\Phi_0}(\CC^n) \to H_{\Phi_0}(\CC^n)
\]
is bounded if and only if
\begin{equation}
\label{A,lambda,inequality}
    4|\gamma|^2 \|A\| \leq 1-|\gamma|^2,\quad \gamma = \frac{1}{1-2\lambda},
\end{equation}
\begin{equation}
\label{c,d,condition}
    \Re(\gamma\overline{c}\cdot y + (4\gamma A \overline{c}-d)\cdot\ybar ) = 0,\quad \text{when }\,y\in\CC^n,\ (1-|\gamma|^2)y = 4|\gamma|^2 A\ybar.
\end{equation}
Moreover, $\Top(e^q)$ is compact on $H_{\Phi_0}(\CC^n)$ if and only if the inequality \eqref{A,lambda,inequality} is strict.
\end{thm}

\medskip

\noindent
\emph{Remark. }In the special case where $A = \mu I$, $\mu\in\CC$, there is a more explicit form of the condition \eqref{c,d,condition}. When $\mu=0$, \eqref{c,d,condition} reduces to
\[
    c = \gamma d,\quad\text{if}\ \,|\gamma|=1.
\]
When $\mu\neq 0$, we obtain by a direct calculation that \eqref{c,d,condition} holds precisely when
\[
    ie^{-\frac{i\theta}{2}}(\overline{\gamma}c+4\gamma\mu\overline{c}-d) \in\RR^n,\quad \text{if}\ \,\mu = \frac{1-|\gamma|^2}{4|\gamma|^2} e^{i\theta},\ \theta\in[0,2\pi).
\]

%    Insert the bibliography data here.

\end{document}